\documentclass[a4paper, oneside, 12pt]{article}
\RequirePackage{amsmath,amsfonts,amssymb}

\usepackage[english]{babel}          
\usepackage[latin1]{inputenc} % Umlaute direkt hier im Source
\usepackage[T1]{fontenc}      % Schicke Umlaute
\usepackage{amsmath}  % Mathe-Improvements
\usepackage{amsfonts} % Mathe-Fonts
\usepackage{amssymb}  % Mathe-Symbole
\usepackage{amsthm}
\usepackage{fancyhdr} % Schicke Header
\usepackage{wasysym}  % Zus. Symbole, z.b. \checked
\usepackage{graphicx} % Einbindung von eps-Grafiken
\usepackage{epstopdf} % Wandelt fr pdflatex automatisch die .eps-Files
\usepackage{ae}
\usepackage{makeidx}  % Index am Ende
\usepackage{epic}
\usepackage{mathrsfs}

\newcommand{\Z}{\mathbb{Z}}

\newcommand{\R}{\mathbb{R}}
\newcommand{\C}{\mathbb{C}}

\renewcommand{\ge}{\geqslant}

\fancypagestyle{plain}{%
   \fancyhead{}
   
}

\newcommand*\mcE{\mathcal{E}}

\newcommand*\mcL{\mathcal{L}}

\newcommand*\mcR{\mathcal{R}}

\newcommand*\mfs{\mathfrak{s}}

\newcommand*\mcV{\mathcal{V}}

\newcommand*\mfu{\mathfrak{u}}

\newcommand*\mfz{\mathfrak{z}}

\newcommand*{\ra}{\rightarrow}

\usepackage{enumerate}
\usepackage[all]{xy}  %Das Packet
\CompileMatrices      %�ersetze die Zeichnungen nur einmal
\SelectTips{cm}{}     %Pfeilspitzten ausrichten

\newtheoremstyle{example}{\topsep}{\topsep}%
     {\itshape}%         Body font
     {}%         Indent amount (empty = no indent, \parindent = para indent)
     {\bfseries}% Thm head font
     {.}%        Punctuation after thm head
     {\newline}%     Space after thm head (\newline = linebreak)
     {\thmname{#1}\thmnumber{ #2}\thmnote{ #3}}%         Thm head spec

\theoremstyle{example}
\newtheorem{theorem}{Theorem}[section]
\newtheorem{proposition}[theorem]{Proposition}
\newtheorem{lemma}[theorem]{Lemma}

\newtheorem{definition}[theorem]{Definition}

\newtheoremstyle{remark}{\topsep}{\topsep}%
     {}%         Body font
     {}%         Indent amount (empty = no indent, \parindent = para indent)
     {\bfseries}% Thm head font
     {.}%        Punctuation after thm head
     {\newline}%     Space after thm head (\newline = linebreak)
     {\thmname{#1}\thmnumber{ #2}\thmnote{ #3}}%         Thm head spec

\theoremstyle{remark}
\newtheorem{remark}[theorem]{Remark}
\newtheorem{remarks}[theorem]{Remarks}
\newtheorem{notation}[theorem]{Notation}

\begin{document}
%************************************************************************************

\nocite{Arvanitoyeorgos03}
\nocite{Berard86}
\nocite{BGM71}
\nocite{Buser92}
\nocite{Chavel84}
\nocite{Gilkey95}
\nocite{Hatcher02}
\nocite{Docarmo92}
\nocite{Schueth03}
\nocite{Karcher89}

\title{isospectral metrics on projective spaces}
\author{Ralf R\"uckriemen}
\date{\today}

\begin{titlepage}
\maketitle

\begin{center}
Humboldt-Universit\"at zu Berlin \\
Mathematisch-Naturwissenschaftliche Fakult\"at II \\
Institut f\"ur Mathematik \\
\end{center}

\vspace{6.5cm}

\begin{tabular}{ll}
turned in by: & Ralf R\"uckriemen \\
born: & 31 december 1981	 in Berlin\\
supervisor: &Prof. Dr. Schueth \\
\end{tabular}

\vspace{1.5cm}

I would like to thank my supervisor Prof. Dr. Schueth for her extensive support, she always took her time to look at my work and made a lot of helpful and necessary corrections. I also want to thank my fellow students Julia Becker-Bender and Sebastian Wuttke for proofreading my thesis.

\end{titlepage}

\vspace*{1cm}

\begin{center}
\LARGE isospectral metrics on projective spaces
\end{center}

\vspace{1.5cm}

\begin{abstract}
We construct isospectral non isometric metrics on real and complex projective space. We recall the construction using isometric torus actions by Carolyn Gordon in chapter 2. In chapter 3 we will recall some facts about complex projective space. In chapter 4 we build the isospectral metrics. Chapter 5 is devoted to the non isometry proof of the metrics built in chapter 4. In chapter 6 isospectral metrics on real projective space are derived from metrics on the sphere.
\end{abstract}

\tableofcontents

\pagebreak
\section{Introduction}
%********************************************************************

The aim of this thesis is to extend the list of examples of
isospectral Riemannian manifolds by projective spaces. \\
Isospectral here means that the set of eigenvalues, including
multiplicities, of the Laplace operator acting on functions stays
the same. Note that the manifolds are also required to be non
isometric as isospectrality would be trivial otherwise. \\
At first we will elaborate a bit on why the problem of isospectrality is studied, 
where it belongs in a broader picture of differential geometry and what has been 
found out so far. \\
We will assume all manifolds to be compact, connected and without boundary throughout this paper. 
Let $f$ be a function on the manifold $M$, let $x_i$ be local coordinates, let $g=(g_{ij})$ denote the metric of the manifold in these coordinates and $g^{ij}$ the inverse matrix. Then the Laplace operator is defined to be \\
\begin{equation*}
\Delta f := -det(g_{ij})^{-\frac{1}{2}}\sum_{i,j}\frac{\partial}{\partial x_j}det(g_{ij})^{\frac{1}{2}}g^{ij}\frac{\partial}{\partial x_i}f 
\end{equation*}
If the metric is the standard metric $g_{ij}=\delta_{ij}$ this simplifies to the familiar Laplacian $\Delta f :=-\sum_{i} \frac{\partial^2}{\partial^2 x_i}f$.
Note that the Laplace operator is self adjoint, elliptic and positive definite.
It can be proven that the spectrum forms a discrete series starting at $\lambda_0=0$ tending 
to infinity. The multiplicity of each eigenvalue is finite. The eigenfunctions can be used to form an orthonormal base
 of all $C^{\infty}(M)$ functions. 
See for example \cite{Berard86} or \cite{Chavel84} for an introduction to the topic. \\
An explicit calculation of the spectrum of a manifold is possible only for very few special manifolds, 
principally the torus and the sphere with standard metric \cite{BGM71}. \\
Nevertheless there is a very close connection between a manifold and its spectrum. \\
$\lambda_0=0$ is always an eigenvalue with multiplicity one, its eigenfunctions are the constant funtions on $(M,g)$. 
There are several theorems estimating the first eigenvalue $\lambda_1$ from conditions on the curvature of the manifold and vice versa. An example is Lichnerowicz theorem, it states that for any closed manifold of dimension $n$ with $Ric \ge k(n-1)$ we have $\lambda_1 \ge nk$, here k is any positiv real number (see \cite{Chavel84}). \\
The spectrum of a manifold also determines a set of spectral invariants. A given eigenvalue spectrum requires the 
manifold to have a certain dimension, a certain volume and it fixes several curvature terms, the first being the total scalar 
curvature. An elegant way to see this is the asymptotic development of Minakshisundaram-Pleijel (see for example \cite{BGM71}) using a fundamental solution of the heat equation. This fundamental solution can be stated explicitly if the eigenvalues and eigenfunctions of the Laplace operator are known. Its asymptotic development for t tending to zero consists of a series of coefficients depending on the curvature of the manifold in a universal way. These coefficients are spectrally determined Riemannian invariants, the so-called heat invariants. \\
Using these and other similar results it can be shown that some special types of manifolds are spectrally determined. Manifolds 
of dimension two with zero curvature and round spheres up to dimension 6 are spectrally determined up to isometry. \cite{BGM71} \\
However, in general manifolds are not spectrally determined. The first examples of isospectral, non locally isometric manifolds were given by J. Milnor in 1964, a pair of tori in dimension 16 (see \cite{Milnor64}).
Various other examples have been found since then. They include
tori of lower dimensions, spheres \cite{Gordon01}, \cite{Schueth01b}, products of spheres and/or
tori \cite{Schueth99}, \cite{Schueth01a} and certain Lie groups \cite{Schueth01a}. We will construct examples with real and complex projective spaces in this paper. \\
Basically two techniques for the construction of isospectral metrics are known. The first is the so-called 'Sunada method'. The idea is to start with a Riemannian manifold and its isometry group. Sometimes it is possible to choose two subgroups of the isometry group with certain properties which guarantee that the manifolds obtained by dividing the Riemannian manifold by these subgroups are isospectral. The other method, found by Carolyn Gordon, uses a torus acting on two Riemannian manifolds. If the quotients of the manifolds by any subtorus are isospectral, then the original manifolds are isospectral, too. \\
In this thesis we will use the second method in a version from Schueth's paper \cite{Schueth01b}. For complex projective spaces, the main purpose of this thesis, an adaptation of Schueth's technique to this setting is needed, although the key principle of the non isometry proof rests the same. \\
We will explain the construction in a general setting in chapter 2. In the third chapter we define projective spaces and gather some facts about complex projective space. We will adapt the construction to our special case and build the isospectral metrics on $\C P^n$ in the fourth chapter. We will start out on the sphere and show that the construction of the metrics is in an appropriate sense compatible with the Hopf fibration. In the original construction a certain horizontality condition is destroyed by the Hopf fibration. However, there exists a general workaround to this problem. In chapter 5 we will prove that the metrics are not isometric . Due to the workaround used in chapter 4 we will need several lemmata to prove the metrics still behave on complex projective space in the specific way needed for the general nonisometry theorem from Schueth in \cite{Schueth01b}.
Isospectral metrics on real projective space are an almost direct corollary of those on spheres, as they are obtained by factoring out a discrete subgroup. This will be treated in chapter 6. 

\pagebreak
\section{Construction of isospectral metrics}
%******************************************************

\begin{notation}
\label{mhut}
Let {$T$} be a torus with a group structure making it a compact,
abelian Lie group. Let $\mfz$ denote its Lie algebra. If {$T$}
acts smoothly and effectively by isometries on a Riemannian
manifold $(M,g)$, then we denote by $\hat M$ the union of those
orbits on which $T$ acts freely. This action of {$T$} gives $\hat
M$ the structure of a principal $T$-bundle. By $g^T$ we denote the
unique Riemannian metric induced on $\hat M \slash T$ such that
the projection from $(\hat M, g)$ is a Riemannian submersion.
\end{notation}

\begin{theorem}
\label{isomF}
\cite{Schueth01b}
Let T be a torus which acts effectively on two compact Riemannian
manifolds $(M,g)$ and $(M',g')$ by isometries. For each subtorus
$W \subset T$ of codimension one, suppose that there exists a
T-equivariant diffeomorphism $F_W : M \ra M'$ which satisfies
$F^*_Wdvol_{g'}=dvol_g$ and induces an isometry $\bar F_W$ between
the quotient manifolds $(\hat M /W, g^W)$ and $(\hat M' /W,
g'^W)$. Then $(M,g)$ and $(M',g')$ are isospectral; if the
manifolds have boundary, then they are Dirichlet and Neumann
isospectral.
\end{theorem}

This theorem is a slight variation of Carolyn Gordon's isospectrality constructions through torus actions. The basic idea of the proof is as follows. One considers the Hilbert space $H^{1,2}(M,g)$. This is the completion of $C^{\infty}(M)$ with respect to the norm:
\begin{equation*} 
||f||^2_{H^{1,2}(M,g)}=\int_M |f|^2 dvol_g + \int_M||df||^2_g dvol_g 
\end{equation*}
One can then give a variational characterisation of the eigenvalues through the Rayleigh quotient defined as: 
\begin{equation*} 
\mcR(f)=\int_M||df||^2_g dvol_g \Big\slash \int_M |f|^2 dvol_g
\end{equation*}  
The main part of the proof is constructing an isometry between the two Hilbert spaces $H^{1,2}(M,g)$ and $H^{1,2}(M',g')$ that preserves $L^2$-norms. The existence of this isometry implies that the eigenvalues of $(M,g)$ and $(M',g')$ are equal. See \cite{Schueth01b} for a complete proof.

\pagebreak

\begin{notation}
\label{VF*}
\begin{tabular}{r p{12.5cm}}
\textup{1.} & Let $T=(\R / 2 \pi \Z) \times (\R / 2 \pi \Z)$ be the standard 2-torus and let $\mcL= 2 \pi \Z \times 2 \pi \Z$ be the associated lattice. Let $\mcL^*$ denote the dual lattice. \\
&\\
\textup{2.} & Let $exp: \mfz \ra T$ be the standard cover, i.e. the exponential map from the Lie algebra $\mfz$ to the Lie group $T$. For $Z \in \mfz$ we denote by $Z^*$ the vector field $p \mapsto \frac{d}{dt} \big|_{t=0}exp(tZ)p$ on $M$. This is the infinitesimal flow induced by the torus action. We will denote the set of all $Z^*$ by $\mfz^*$.\\
\end{tabular}
\end{notation}

\begin{definition}
\label{admissible}
Let $\lambda$ denote a 1-form on $M$ with values in $\mfz$. 
\begin{enumerate}
\item We call a 1-form $\lambda$ admissible iff
\begin{enumerate}
\item it is T-invariant, and 
\item it is horizontal. That is, $\lambda$ vanishes on the tangent spaces of the orbits of the $T$ action on
$M$ or put in an equation we have $\lambda(U)=0 \textit{ } \forall U \in \mcV(M)$
with $U \in span\{Z_1^*, Z_2^* \}$. Here $\mcV(M)$ denotes the set of all vector fields on $M$. 
\end{enumerate}
\item For any admissible 1-form $\lambda$ we define a Riemannian metric
$g_{\lambda}$ on $(M, g_0)$ by: \\
\begin{equation*}
g_{\lambda}(X,Y) := g_0(X + \lambda(X)^*, Y + \lambda(Y)^*)
\end{equation*}
\end{enumerate}
\end{definition}

\begin{remarks}
For any admissible 1-form $\lambda$ the map $X \mapsto \lambda(X)^*$ is two-step nilpotent. This implies $dvol_{g_{\lambda}}=dvol_{g_0}$, that is the new metrics have the same volume element as $g_0$.
\end{remarks}

\begin{proposition}
\label{riemsubmersion}
Let the torus $T$ act isometrically on $(M,g_0)$ and let $\lambda$ denote an admissible 1-form. \\
Then $T$ acts isometrically on $(M, g_{\lambda})$ and the Riemannian submersion metrics $g_0^T$ and $g_{\lambda}^T$ on $\hat{M} \slash T$ are equal.
\end{proposition}
\begin{proof}
As $\lambda$ is T-invariant and the metric $g_0$ is T-invariant the metric $g_{\lambda}$ is T-invariant, too; thus the torus acts isometrically on $(M, g_{\lambda})$. Thus there exist unique metrics $g_0^T$ and $g_{\lambda}^T$ on $\hat{M} \slash T$ such that both  projections are Riemannian submersions. \\
We have that $g_{\lambda}$ restricts to the same metric as $g_0$ on the space spanned by $Z_1^*$ and $Z_2^*$ as $\lambda$ vanishes on these vector fields. Let $\pi$ denote the projection from $\hat{M}$ to $\hat{M} \slash T$. Then $ker(d \pi_p)$ is a subspace of $T_p\hat{M}$. A submersion is called a Riemannian submersion if the isomorphic map $d \pi_p:  (ker(d \pi_p))^{\bot_g} \ra T_{\pi(p)}(\hat{M} \slash T)$ is an isometry. \\
That means the assertion is equal to saying that the scalar product $g_0$ restricted to vectors that are $g_0$-orthogonal to $\mfz^*$ is the same scalar product as $g_{\lambda}$ restricted to vectors that are $g_{\lambda}$-orthogonal to $\mfz^*$. By definition of $g_{\lambda}$ we have that if $X$ is $g_0$-orthogonal to $\mfz^*$, then $X-\lambda(X)^*$ is $g_{\lambda}$-orthogonal to $\mfz^*$ and $d \pi _p(X)=d \pi_p(X-\lambda(X)^*)$. Let $X$ and $Y$ be two $g_0$-horizontal vector fields. Then we have:
\begin{eqnarray*}
&g_{\lambda}(X-\lambda(X)^*, Y-\lambda(Y)^*) \\
=&g_{0}(X-\lambda(X)^*+\lambda(X-\lambda(X)^*)^*, Y-\lambda(Y)^*+\lambda(Y-\lambda(Y)^*)^*) \\
=&g_{0}(X, Y) \\
\end{eqnarray*}
as the map $X \mapsto \lambda(X)^*$ is linear and two-step nilpotent. This implies that the metrics $g_0^T$ and $g_{\lambda}^T$ are equal.
\end{proof}

\begin{theorem}
\cite{Schueth01b}
\label{isomlambda}
Let $\lambda$, $\lambda'$ be two admissible $\mfz$-valued 1-forms
on M. Assume: \\
\begin{enumerate}
\item[$(I)$] For every $\mu \in \mcL^*$ there exists a T-equivariant $F_{\mu}
\in Isom(M,g_0)$ which satisfies $\mu \circ \lambda = F_{\mu}^*(\mu \circ \lambda')$.
\end{enumerate}
Then $(M, g_{\lambda})$ and $(M', g_{\lambda'})$ are isospectral.
\end{theorem}
This theorem is an application of theorem \ref{isomF}. For every subtorus $W$ of codimension one in $T$ choose $\mu \in \mcL^*$ such that the Lie algebra of W is exactly the kernel of $\mu$. The isometry $F_{\mu}$ is then the map $F_W$ from theorem \ref{isomF}. \\

We will now describe how this machinery is realized on odd dimensional spheres. We will adapt this to complex projective spaces in later chapters.

\begin{notation}
Let 
\begin{equation*}
SU(m)=\{ A \in Gl(m,\C) \mid A \bar A ^{T}=1, det(A)=1 \}
\end{equation*}
denote the Lie group of unitary complex matrices and let 
\begin{equation*}
\mfs \mfu (m)= \{ A \in M(m, \C) \mid A + \bar A ^{T}=0, tr(A)=0 \}
\end{equation*}
denote its Lie algebra. We remind that the torus $T$ is also a Lie group and its Lie algebra is denoted by $\mfz$. See \cite{Arvanitoyeorgos03} for an introduction to Lie groups and Lie algebras.
\end{notation}

\begin{remark}
\label{torusiso}
We construct admissible 1-forms $\lambda$, $\lambda'$ on $S^{2n+1}$ using suitable linear
maps $j,j': \mfz \ra \mfs \mfu (n-1)$. We write $j_Z$ for $j(Z)$ for the sake of convenience. With the inclusion $SU(n-1) \times T = SU(n-1) \times U(1) \times U(1) \subset U(n+1)$ we get a canonical action of $SU(n-1)$ and of {$T$} on $\C^{n+1}$. These two actions commute and preserve the sphere $S^{2n+1}$ embedded in $\C^{n+1}$. \\ 
Let $Z_1$ and $Z_2$ denote the standard base of $\mfz \cong \R^2$. Then the action of the torus on a point $p=(q,r,s) \in \C^{n-1} \times \C \times \C \cong \C^{n+1}$ is given by \\
\begin{equation*}
exp(aZ_1 +bZ_2): (q,r,s) \mapsto (q, e^{ia}r, e^{ib}s)
\end{equation*}
Both the torus action and the action of $SU(n-1)$ induce isometries on the sphere $(S^{2n+1}, g_{eucl})$.
The 1-forms $\lambda$, $\lambda'$ are then defined to be $\lambda^k(X)_{p}:=\langle j_{Z_k}p, X \rangle$, for $k=1,2$, and similarly for $\lambda'$. Here $j_{Z_k}p$ is understood as $(j_{Z_k}q, 0, 0)$, if $p=(q,r,s)$, that is, the last two components of $p$ and $X$ play no role. \\
The scalar product $\langle , \rangle$ denotes the standard hermitian scalar product on $\C^{n+1}$. Explicitly we have $\langle X,Y \rangle := \sum_i Re (X_i \overline{Y_i})$. The associated metric on $S^{2n+1}$ is the standard euclidian metric, it is denoted by $g_{eucl}$.\\
The operation of a matrix in $\mfs \mfu (n-1)$ on a point of the the sphere $S^{2n+1}$ is defined via the Lie group $SU(n-1)$. Consider a curve $\gamma(t)$ in $SU(n-1)$ with $\gamma(0)=id$ and $\dot{\gamma}(0) =j_{Z_k} \in \mfs \mfu (n-1)$. Then $\gamma(t)p$ is a curve on $S^{2n+1}$, its differential at zero lies in $T_pS^{2n+1}$ and is written as $j_{Z_k}p$. This definition is independent of the curve $\gamma$ chosen, the curve $\gamma(t)=exp(tj_{Z_k})$ would be a possible choice. \\
\end{remark}

\begin{remark}
\label{deflambda}
The 1-form $\lambda$ on the sphere defined as $\lambda^k(X)_{p}:=\langle j_{Z_k}p, X \rangle$, for $k=1,2$, is admissible as in definition \ref{admissible}. In fact it is $T$-invariant as $\lambda$ depends only on the $q$ components of $p$ and $T$ acts only on the $r$ and $s$ components. We also have $\lambda(U)=0$ for all $U \in \mfz ^*$ since $j_Zp=(j_Zq,0,0)$ is orthogonal to $(0,ir,0)$ and $(0,0,is)$. 
\end{remark}

\begin{definition}
\label{defj}
Let $j,j': \mfz \cong \R^2 \ra \mfs \mfu (m)$ be two linear maps.  
\begin{enumerate}
\item We call $j$ and $j'$ isospectral, denoted $j \sim j'$, iff
$\forall Z \in \mfz \exists A_Z \in SU(m)$ such that
$j'_Z=A_Zj_ZA_Z^{-1}$. This implies that for each $Z \in \mfz$, $j_Z$ and $j'_Z$ have the same eigenvalues.

\item Let $Q: \C^{m} \ra \C^{m}$ denote complex conjugation. We call
$j$ and $j'$ equivalent, denoted $j \cong j'$, iff there exists $A
\in SU(m) \cup SU(m)\circ Q$ and $\Psi \in \mcE$ such that
$j'_Z=Aj_{\Psi(Z)}A^{-1}$ for all $Z \in \mfz$. Here $\mcE$ is the group of all automorphisms of $\mfz$ which preserve the set $\{ \pm Z_1, \pm Z_2 \}$.

\item We say $j$ is generic iff no nonzero element of $\mfs \mfu
(m)$ commutes with both $j_{Z_1}$ and $j_{Z_2}$.

\end{enumerate}
\end{definition}

The first property will, as its name suggests, guarantee isospectrality of the constructed metrics. We will use the matrices $A_Z$ to construct the isometries $F_{\mu}$ from condition $(I)$ in theorem \ref{isomlambda}.
Non equivalence and genericity will be used in the non isometry proof. Note that any isometry of $\C P^{m-1}$ is induced by a matrix $A$ as used in the non equivalence definition. We will give the complete isometry group of $\C P^{m-1}$ in the appropriate chapter.\\

We will check non equivalence by evaluating the term {$tr((j_{Z_1}^2+j_{Z_2}^2)^2)$}.
\begin{lemma}
If 
\begin{equation*}
tr((j_{Z_1}^2+j_{Z_2}^2)^2) \neq tr((j'^2_{Z_1}+j'^2_{Z_2})^2)
\end{equation*}
then $j$ and $j'$ are not equivalent.
\end{lemma}
\begin{proof}
We will use an indirect proof. If $j \cong j'$, we have:
\begin{eqnarray*}
tr((j'^2_{Z_1}+j'^2_{Z_2})^2) =& tr(((Aj_{\Psi(Z_1)}A^{-1})^2+(Aj_{\Psi(Z_2)}A^{-1})^2)^2) \\
=& tr((Aj^2_{\Psi(Z_1)}A^{-1}+Aj^2_{\Psi(Z_2)}A^{-1})^2) \\
=& tr((A(j^2_{\Psi(Z_1)}+j^2_{\Psi(Z_2)})A ^{-1})^2) \\
=& tr((A(j^2_{Z_1}+j^2_{Z_2})A ^{-1})^2) \\
=& tr(A(j^2_{Z_1}+j^2_{Z_2})^2A ^{-1}) \\
=& tr((j^2_{Z_1}+j^2_{Z_2})^2)
\end{eqnarray*}
Turning this the other way around, we get that if {$tr((j_{Z_1}^2+j_{Z_2}^2)^2)\neq tr((j'^2_{Z_1}+j'^2_{Z_2})^2)$}, then $j$ and $j'$ are not equivalent. \\
\end{proof}

What we now need are pairs of maps $j$, $j'$ that are isospectral, non equivalent and
generic. The necessity of their existence is our main dimension barrier. We want $SU(m)$ to act isometrically on the manifold and an additional 2-torus whose isometric action commutes with $SU(m)$. \\
A simple calculation shows no such maps exist for $m=1,2$. On the other hand there exist not only pairs but even continuous families of such maps $j(t)$ for any dimension larger or equal than 3.  The proof can be found in \cite{Schueth01b}, Prop 3.2.6. or in \cite{Schueth01a}, Prop 3.6. An explicit example for such a family with $m=3$
is: \\

\begin{eqnarray*}
j_{Z_1}(t):=&
\begin{pmatrix}
4i & 0 & 0 \\
0 & i & 0 \\
0 & 0& -5i
\end{pmatrix} 
j_{Z_2}(t):=&
\begin{pmatrix}
4i & 0 & 0 \\
0 & i & 0 \\
0 & 0& -5i
\end{pmatrix} 
\end{eqnarray*}
A direct computation shows that 
\begin{equation*}
 det(\lambda \cdot Id-(aj_{Z_1}(t)+bj_{Z_2}(t)))
= \lambda^3+(3a^2+21b^2)\lambda -3ia^2b-20ib^3
\end{equation*}
is independent of $t$. In other words for any $Z=aZ_1+bZ_2$ we have that all $j_Z(t)$ are conjugate to each other because they have the same eigenvalues. On the other hand 
\begin{equation*}
tr((j_{Z_1}(t)^2+j_{Z_2}(t)^2)^2)=1038+108\cos^2(t)
\end{equation*}
clearly depends on $t$. We also have $j(t)$ is generic for all $t$ with $\sin(t) \neq 0$. \\

\pagebreak
\section{Projective spaces in general and $\C P^n$ in particular}
%***********************************************************

\begin{definition}
For any field $F$ the set of all lines in the vector space $F^{n+1}$ is called the n-dimensional projective space of that field. If $F$ is the field of real, complex or quaternionic numbers, this is a manifold. It can be written as the quotient space $F^{n+1} \slash F^*$, where $F^*=F \backslash \{0\}$.
\end{definition}

\begin{remark}
Projective spaces have some interesting geometric properties. They were one of the first examples of non euclidean geometry to be discovered. There is no notion of lines being parallel, instead we have a theorem stating that any two distinct lines in a common plane intersect in exactly one point. See for example \cite{coxeter74} for an introduction to projective geometry.
\end{remark}

\begin{remark}
In the case of real, complex or quaternionic projective space the quotient space can be 'simplified' by dividing through the norm. We restrict to the unit sphere, in the complex case we get $\C P^n= S^{2n+1} / S^1$. This is known as the Hopf fibration. We will denote by $\Pi$ the projection map $\Pi :S^{2n+1} \ra \C P^n$. A point $[p] \in \C P^n$ is then the equivalence class of a point $p \in S^{2n+1} \subset \C^{n+1} \cong \R^{2n+2}$ modulo multiplication with $\tau \in S^1 \subset \C$. Similarly real projective space can be seen as a sphere with opposite points identified.
\end{remark}

\begin{remark}
\cite{Hatcher02} Complex projective space can be given the structure of a CW-complex. This is particulary useful to compute its (real) homology. We can obtain $\C P^n$ also as a quotient space of the disk $D^{2n}$ under the identifications $p \sim \tau p$ for $p \in \partial D^{2n} =S^{2n-1}$ and $\tau$ a complex number with norm one. We have $\C P^n= S^{2n+1} / S^1$, if we choose representatives in $S^{2n+1}$ where the first coordinate is real and non negative we get a one-to-one correspondence between points in $\C P^n$ with first coordinate non zero and points in the interior of the disk $D^{2n}$. The border of $D^{2n}$ corresponds to the points in $S^{2n+1}$ with first coordinate zero, so we still need the identification $p \sim \tau p$ there. This description implies that $\C P^n$ can be obtained from $\C P^{n-1}$ by attaching a $2n$-dimensional cell. Inductively we obtain that the cell structure of $\C P^n$ consists of exactly one cell in every even dimension up to $2n$. \\
Thus we have proven that $\C P^n$ seen as a CW-complex has no two cells in adjacent dimensions. This means that all homology border maps are zero and the cells are in one-to-one correspondence with the generators of the homology groups. Thus we get:
\begin{equation*}
H_i(\C P^n) \cong  \begin{array}{cl} \Z & \textit{for } i=0,2,4,\hdots , 2n \\ 0 & otherwise  \end{array}
\end{equation*}
\end{remark}

\begin{remark}
Complex projective space is a complex manifold, one atlas is the set of charts $\varphi_j : \C^n \ra \C P^n$ where every point in $\C^n$ is identified with the point $[ \hdots, 1, \hdots]$ in $\C P^n$ with the $1$ on the j-th position, $j$ ranges form $0$ to $n$ here. The transitions between two maps are holomorphic. The complex structure $J: T_{[p]} M \ra T_{[p]} M$, satisfying $J^2= -id|_{T_{[p]} M}$, is given by multiplication with the complex number $i$.\\
\end{remark}

\begin{remark}
\label{FubiniStudy}
With the realization of complex projective space as a quotient space we also get a naturally induced metric. We start with the euclidian metric $g_{eucl}$ on the sphere (induced from the euclidian metric on $\R ^{2n}$) and then take the unique metric such that the Hopf fibration is a Riemannian submersion. For complex projective space this metric is called Fubini-Study metric, denoted by $g_{FS}$. Within the domain of one chart it can be written explicitly in the following way. \\
Let $[p]=[p_0, p_1, \hdots,p_n ]$ denote a point in $\C P^n$. We will use the chart $\varphi_0$ so we have $p_0=1$. Let $\bar p_i$ denote the complex conjugate, then $g_{FS}$ can be written as
\begin{eqnarray*}
g_{FS}=(g_{i\bar{j}})_{i\bar{j}}=& (\frac{\partial^2}{\partial p_i \partial \bar{p}_j} log(1+|p|^2))_{i\bar{j}} \\
=& \sum\limits_i \frac{dp_i \wedge d\bar{p}_i}{1+|p|^2}- \sum\limits_{i,\bar{j}}\frac{p_i dp_i \wedge \bar{p}_j d\bar{p}_j}{(1+|p|^2)^2}
\end{eqnarray*}
\end{remark}

\begin{remark}
$\C P^n$ admits K\"ahler metrics, we will prove here that the Fubini-Study metric is an example. We have to show that the associated K\"ahler-form $\omega(X,Y):=g_{FS}(JX,Y)$ is closed. This can be seen by a straightforward calculation of $d\omega$ in one of the charts. We will denote by $\partial=\sum\limits_i\frac{\partial}{\partial p_i}$ and $\bar{\partial}=\sum\limits_j\frac{\partial}{\partial \bar{p}_j}$ the familiar operators in the special case of zero forms, i.e. functions. Then we have 
\begin{eqnarray*}
\omega=& i \partial \bar{\partial}log(1+|p|^2) \\
d\omega =& di \partial \bar{\partial}log(1+|p|^2) \\
=&i(\partial + \bar{\partial}) \partial \bar{\partial}log(1+|p|^2) \\
=&i \bar{\partial} \partial \bar{\partial}log(1+|p|^2) \\
=&-i \partial \bar{\partial}  \bar{\partial}log(1+|p|^2) \\
=&0
\end{eqnarray*}
Here we used $d=\partial + \bar{\partial}$, $\partial^2=0$, $\bar{\partial}^2=0$ and $\partial \bar{\partial} +  \bar{\partial}\partial =0$.
\end{remark}

\begin{remark}
\label{isometrie}
Any unitary map $U: \C^{n+1} \ra \C^{n+1}$ induces an isometry on the sphere $S^{2n+1}$ embedded in $\C^{n+1}$. The Hopf circle of a given point $p \in S^{2n+1}$ consists of the points $\{ \tau p \mid \tau \in \C, |\tau |=1 \}$. $U$ and $\tau \cdot id$ commute because the center of the group of unitary matrices is given exactly by complex multiples of the identity. We thus have that $U$ maps Hopf circles to Hopf circles, therefore it gives an isometry of $\C P^{n}$. Complex conjugation also maps Hopf circles to Hopf circles. Let $Q$ denote complex conjugation. Then we have additional isometries of $\C P^{n}$ of the form $U \circ Q$. In fact all isometries of complex projective space arise this way, see \cite{Karcher89}. \\
This construction does not give the isometry group of $\C P^{n}$ directly though because not all maps give rise to different isometries. More accurately, if $U=\tau U'$ for a $\tau \in S^1$, then $U$ and $U'$, as well as $U \circ Q$ and $U' \circ Q$, give rise to the same isometry on $\C P^{n}$. \\
\end{remark}

\begin{remark}
\label{complexaction}
We can now take a closer look at the actions of $SU(n-1)$ and the torus $T$ on $\C P^{n}$ as in remark \ref{torusiso}. We have $SU(n-1) \times T = SU(n-1) \times U(1) \times U(1) \subset U(n+1)$, thus both consist of isometries on complex projective space by the remark above. Both $SU(n-1)$ and the torus act effectively on $\C P^n$. Note however, that the combined action of $SU(n-1) \times T$ is not effective, i.e. the isometries are not all pairwise different. Suppose $A, A'$ are two matrices in $U(n+1)$ of the form $B \times e^{ia} \times e^{ib} \in SU(n-1) \times T \subset U(n+1)$ with $a,b \in [ 0, 2 \pi [$. They induce the same isometry on $\C P^{n}$ if there exists a complex number $\tau \in S^1$ with $A=\tau A'$. This implies $B= \tau B'$ and as $B$ and $B'$ have determinant one we get $\tau=e^{\frac{2\pi i}{n-1}k}$ for an integer $k=0,1,\hdots , n-2$. Thus for every element in $SU(n-1) \times T$ there exists a family of elements that induce the same isometry on $\C P^{n}$. However, this family is finite. That means that the action of $SU(n-1)\times T$ is still almost effective, which is sufficient in our setting. \\ 
\end{remark}

\begin{remark}
\label{Hopfcommute}
The Hopf action on the sphere $S^{2n+1}$ embedded in $\C^{n+1}$ is given by matrices of the form $\tau id$ with a complex numbers $\tau$ of norm one. As the complex multiples of the identity are exactly the center of $U(n+1)$ we get that the Hopf action commutes with the action of $U(n+1)$ and a fortiori with the actions of the torus $T$ and of $SU(n-1)$.
\end{remark}

\pagebreak
\section{Isospectral metrics on $\C P^{n}$}
%****************************************************************

\begin{notation}
In the following two chapters we will frequently jump between objects on the sphere and objects on complex projective space. We will continue to denote a point on the sphere embedded in $\C^{n+1}$ by $p=(q,r,s)$, where $q$ stands for the first $n-1$ complex coordinates and $r$ and $s$ for the last two as in remark \ref{torusiso}. A point in complex projective space, i.e. the equivalence class of a point on the sphere, will be denoted by $[p]=[q,r,s]$. \\
\end{notation}

\begin{remark}
\label{barlambda}
Let the sphere $S^{2n+1}$ be embedded in $\C ^{n+1}$ and let $\tau$ denote a complex number with norm one. Multiplication by $\tau$ defines a map on the sphere that will also be denoted by $\tau : S^{2n+1} \ra S^{2n+1}$.
\label{ubertragen}
Let $X \in \mcV (S^{2n+1})$ with \\
$(\diamondsuit) \qquad X_{\tau p}=\tau_* X_{p} \quad$ \\
for all  $\tau$. These vector fields are invariant under the Hopf action. We additionally assume that $X_{p} \perp ip$, that is we restrict our attention to vector fields which are horizontal with respect to the Hopf fibration. Thus we get an isomorphism between all vector fields on $S^{2n+1}$ satisfying these conditions, that is all $S^1$-invariant Hopf-horizontal vector fields, and all vector fields on $\C P^n$. This isomorphism is given by the differential of the Hopf projection $\Pi: S^{2n+1} \ra \C P^n$. We will denote $S^1$-invariant Hopf-horizontal vector fields on the sphere simply by $X$ and the associated vector fields on $\C P^n$ by $[X]$. \\
Note that the vector fields $p \mapsto j_Z p$ from remark \ref{torusiso} satisfy condition $(\diamondsuit)$ because we have $(j_Z\tau q, 0,0)=\tau(j_Z q,0,0)$ for all $j_Z \in \mfs \mfu (n-1)$. However, they are not orthogonal to the vector field $ip$.
\end{remark}

The next goal is to construct suitable 1-forms $\bar{\lambda}$ on complex projective space. For a given $S^1$-invariant 1-form $\lambda$ on $S^{2n+1}$ we define an associated 1-form on $\C P^n$, denoted by $\bar{\lambda}$, by letting $\bar{\lambda}([X])_{[p]}:={\lambda}(X)_{p}$ where $p$ is any representative of $[p]$ on the sphere and $X$ is the unique $S^1$-invariant Hopf-horizontal lift of $[X]$ to $S^{2n+1}$. The remark above shows that $\bar{\lambda}$ is a well defined 1-form on $\C P^n$. \\

By remark \ref{complexaction} we know that $SU(n-1)$ and the torus $T$ act on $\C P^n$. By remark \ref{Hopfcommute} we know that both these actions commute with the Hopf action. That means the 1-form $\bar{\lambda}$ on $\C P^n$ satisfies $\bar{\lambda}^k([X])_{[p]}=g_{FS}( \Pi_*(j_{Z_k}p), [X] )$ for $k=1,2$. \\
However the 1-form $\bar\lambda$ is in general not admissible as in definition \ref{admissible} even if $\lambda$ is admissible with respect to the T-action on the sphere. The T-invariance carries over to $\C P^n$. We have $j_{Z_k}p=(j_{Z_k}q,0,0)$ is T-invariant by the definition of the torus action, in consequence $\Pi_*(j_{Z_k}p)$ is T-invariant and it follows that $\bar{\lambda}$ is T-invariant.
The horizontality with respect to the T-action, on the other hand, is not transmitted to $\C P^n$. We have to horizontalize. This is possible in a general way as the next proposition shows.

\begin{proposition}
\cite{Schueth03} \label{horizontal} Let $(M, g)$ be a Riemannian manifold on which $T$ acts by isometries. Let $\lambda$ be a $T$-invariant 1-form on $M$. Then the 1-form $\lambda_h$ defined by:
\begin{eqnarray*}
\lambda_h(X):=& ||Z_1^* \wedge Z_2^*||^2 \lambda(X) \\
&- \langle X \wedge Z_{2}^* , Z_{1}^* \wedge Z_{2}^* \rangle \lambda(Z_{1}^*) \\
&- \langle Z_{1}^* \wedge X, Z_{1}^* \wedge Z_{2}^* \rangle \lambda(Z_{2}^*) \\
\end{eqnarray*}
is admissible with respect to the T-action. Here $\{Z_1,  Z_2\}$ is a basis for $\mfz$ and the scalar product on $\Lambda^2 T_p M$ is the one induced by $g$.
\end{proposition}
\begin{proof}
One easily checks that $\lambda_h(Z_k^*)=0$ for  $k=1, 2$. Thus $\lambda_h$ is horizontal, and obviously it is again T-invariant.
\end{proof}

In order to calculate the horizontalized 1-form $\lambda_h$ on $\C P^n$ we will find explicit forms for the horizontalized versions of the vector fields $Z_1^*$ and $Z_2^*$ on $S^{2n+1}$. We will denote by $Z_{h,1}^*$ and $Z_{h,2}^*$ the horizontal parts of $Z_1^*$ and $Z_2^*$ with respect to the Hopf action.  
By the formula in notation \ref{VF*} we get $Z_1^*=(0,i r,0)$ and $Z_2^*=(0,0,i s)$ on the sphere at the point $p=(q,r,s)$. To transfer these vector fields to $\C P^n$ we need to orthogonalize them to the vector field $ip$. Thus we get
\begin{eqnarray*}
(Z_{h,1 }^*)_p=& (0,ir,0) - \frac{\langle (0,i r,0), ip \rangle}{|ip|^2}ip \\
=&(0,i r,0)- |r|^2 ip \\
(Z_{h,2 }^*)_p=& (0,0,i s) - \frac{\langle (0,0,i s), ip \rangle}{|ip|^2}ip \\
=&(0,0,i s) - |s|^2ip 
\end{eqnarray*}
Note that these are the $S^1$-invariant, Hopf-horizontal lifts of the corresponding vector fields $[p] \mapsto \frac{d}{dt}\big|_{t=0}e^{tZ_1}[p]$ and $[p] \mapsto \frac{d}{dt}\big|_{t=0}e^{tZ_2}[p]$ on $\C P^n$. \\
Since the metric $g_{FS}$ on $\C P^n$ is defined as a submersion metric, the scalar products and the norms of these vector fields can be calculated as $\langle Z_{h,1}^*, Z_{h,2}^* \rangle$ and $||Z_{h,k}^*||^2$, $k=1,2$ respectively.
We will need \\
\begin{eqnarray*}
||Z_{h,1}^*||^2_p =& |ir|^2-2|r|^2 \langle (0,ir,0),ip \rangle +|r|^4|ip|^2 \\
=& |r|^2(1-|r|^2) \\
||Z_{h,2}^*||^2_p =& |is|^2-2|s|^2 \langle (0,0,is),ip \rangle +|s|^4|ip|^2 \\
=& |s|^2(1-|s|^2) \\
\langle Z_{h,1}^*, Z_{h,2}^* \rangle _p
=& \langle (0,ir,0), (0,0,is) \rangle - |r|^2 \langle ip, (0,0,is) \rangle \\
& - |s|^2 \langle ip, (0,ir,0) \rangle + |r|^2|s|^2|ip|^2 \\
=& -|r|^2|s|^2
\end{eqnarray*}

\begin{proposition}
\label{defeta}
Let $j: \mfz \ra \mfs \mfu (n-1)$ be a linear map. Then the $\mfz$-valued 1-form $\eta$ on $S^{2n+1}$ defined by 
\begin{equation*}
\eta^k(X)_p :=|q|^2\langle j_{Z_k} q, X_q \rangle - \langle j_{Z_k} q, iq \rangle \langle iq,X_q \rangle \quad \textit{ for } k=1,2
\end{equation*}
is $S^1$-invariant and Hopf-horizontal. The induced 1-form $\bar{\eta}$ on $\C P^n$ is admissible with respect to the $T$-action. Here, $\eta^1(X)$ and $\eta^2(X)$ denote the coordinates in $\mfz$ and $X_q,X_r$ and $X_s$ denote the components of $X \in T_{(q,r,s)}S^{2n+1} \cong \C^{n+1}=\C^{n-1} \oplus \C \oplus \C$ .\\
\end{proposition}

\begin{proof}
We will at first calculate $\bar{\lambda}_h$ on $\C P^n$ in our setting. Here $\lambda$ is the 1-form on $S^{2n+1}$ associated with $j$ as in remark \ref{deflambda}, and $\bar\lambda$ the associated 1-form on $\C P^n$ as introduced in remark \ref{barlambda}. Let $[X] \in T_{[p]}\C P^n$, and let $X$ denote the $S^1$-invariant, Hopf-horizontal lift of $[X]$ to $T_pS^{2n+1}$. Then we have by definition of $\lambda$ on $S^{2n+1}$ and of $\bar\lambda$ on $\C P^n$, by recalling that the metric $g_0$ on $\C P^n$ is a submersion metric arising from the Hopf fibration, and by the formula in proposition \ref{horizontal} :\\
\begin{eqnarray*}
\bar{\lambda}_h^k([X])=& ||[Z_{h,1}^*] \wedge [Z_{h,2}^*]||^2 \bar{\lambda}^k([X])  \\
&- \langle ([X] \wedge [Z_{h,2}^*] , [Z_{h,1}^*] \wedge [Z_{h,2}^*] \rangle \bar{\lambda}^k([Z_{h,1}^*]) \\
&- \langle [Z_{h,1}^*] \wedge [X], [Z_{h,1}^*] \wedge [Z_{h,2}^*] \rangle \bar{\lambda}^k([Z_{h,2}^*]) \\
%%%%%%%%%%%%%%%%
=& ||Z_{h,1}^* \wedge Z_{h,2}^*||^2 \lambda^k(X)  \\
&- \langle X \wedge Z_{h,2}^* , Z_{h,1}^* \wedge Z_{h,2}^* \rangle \lambda^k(Z_{h,1}^*) \\
&- \langle Z_{h,1}^* \wedge X, Z_{h,1}^* \wedge Z_{h,2}^* \rangle \lambda^k(Z_{h,2}^*) \\
%%%%%%%%%%%%%%%%
=& (||Z_{h,1}^*||^2||Z_{h,2}^*||^2 - \langle Z_{h,1}^*, Z_{h,2}^* \rangle ^2 )\langle j_{Z_k}p,X \rangle  \\
&- (\langle X , Z_{h,1}^* \rangle ||Z_{h,2}^*||^2- \langle X, Z_{h,2}^* \rangle \langle Z_{h,1}^*, Z_{h,2}^* \rangle) \langle j_{Z_k}p,Z_{h,1}^* \rangle \\
&- (||Z_{h,1}^*||^2 \langle X, Z_{h,2}^* \rangle - \langle X , Z_{h,1}^* \rangle \langle Z_{h,1}^*, Z_{h,2}^* \rangle)\langle j_{Z_k}p,Z_{h,2}^* \rangle \\
%%%%%%%%%%%%%%%%
=& (|r|^2(1-|r|^2)|s|^2(1-|s|^2) - |r|^4|s|^4)\langle j_{Z_k}p,X \rangle  \\
&- (\langle X , Z_{h,1}^* \rangle |s|^2(1-|s|^2) + \langle X, Z_{h,2}^* \rangle |r|^2|s|^2) \\
& \cdot (\langle j_{Z_k}p,(0,ir,0) \rangle - \langle j_{Z_k}p, |r|^2 ip\rangle) \\
&- (|r|^2(1-|r|^2) \langle X, Z_{h,2}^* \rangle + \langle X , Z_{h,1}^* \rangle |r|^2|s|^2) \\
& \cdot(\langle j_{Z_k}p,(0,0,is) \rangle - \langle j_{Z_k}p, |s|^2ip \rangle) \\
%%%%%%%%%%%%%%%%
=& |r|^2|s|^2(1-|r|^2-|s|^2)\langle j_{Z_k}p,X \rangle  \\
&+ ( |s|^2(1-|s|^2)|r|^2+|r|^2|s|^4) \langle j_{Z_k}p, ip \rangle \langle X , Z_{h,1}^* \rangle\\
&+ ( |r|^2(1-|r|^2)|s|^2+|s|^2|r|^4) \langle j_{Z_k}p, ip \rangle \langle X , Z_{h,2}^* \rangle \\
%%%%%%%%%%%%%%%%
=& |r|^2|s|^2|q|^2\langle j_{Z_k}p,X \rangle  \\
&+ |r|^2|s|^2 \langle j_{Z_k}p, ip \rangle \langle X , Z_{h,1}^* + Z_{h,2}^* \rangle
\end{eqnarray*}
Next we observe that multiplying a 1-form by a $T$-invariant function does not change whether it is admissible or not. Neither the T-invariance nor the horizontality are influenced by it. Note that $[q,r,s] \mapsto |r|^2|s|^2$ is indeed a well-defined, T-invariant funtion on $\C P^n$. We will use this fact to simplify $\lambda_h$ by a factor $|r|^2|s|^2$. As $X$ is the lift of $[X]$ it is Hopf-horizontal and we have $\langle ip, X \rangle =0$. Therefore,
\begin{eqnarray*}
& |q|^2\langle j_{Z_k} p, X \rangle + \langle j_{Z_k} p, ip \rangle \langle Z_{h,1}^*+Z_{h,2}^*, X \rangle \\
=& |q|^2\langle j_{Z_k} q, X_q \rangle + \langle j_{Z_k} q, iq \rangle \langle (0,ir,is) - (|r|^2+|s|^2)ip, X \rangle \\
=& |q|^2\langle j_{Z_k} q, X_q \rangle + \langle j_{Z_k} q, iq \rangle \langle (0,ir,is) , X \rangle \\
&\textit{we will now substract the term } 0=\langle j_{Z_k} q, iq \rangle \langle ip, X \rangle \\
=& |q|^2\langle j_{Z_k} q, X_q \rangle + \langle j_{Z_k} q, iq \rangle \langle (0,ir,is) - ip, X \rangle \\
=& |q|^2\langle j_{Z_k} q, X_q \rangle - \langle j_{Z_k} q, iq \rangle \langle (iq,0,0), X \rangle \\
=& |q|^2\langle j_{Z_k} q, X_q \rangle - \langle j_{Z_k} q, iq \rangle \langle iq,X_q \rangle \\
=& \eta^k(X)
\end{eqnarray*}
Note that $\eta$ is $S^1$-invariant since: 
\begin{eqnarray*}
\eta^k(\tau_*X)_{\tau p} =& |\tau q|^2\langle j_{Z_k}\tau q, (\tau_* X)_{q} \rangle - \langle j_{Z_k} \tau q, i\tau q \rangle \langle i\tau q,(\tau_* X)_q \rangle \\
=& |q|^2\langle \tau j_{Z_k} q, \tau X_q \rangle - \langle \tau j_{Z_k} q, \tau iq \rangle \langle \tau iq,\tau X_q \rangle \\
=& \eta^k(X)_p \\
\end{eqnarray*}
because $\tau$ is an isometry.
Moreover, $\eta$ is Hopf-horizontal because we have:
\begin{equation*}
\eta^k(ip) = |q|^2\langle j_{Z_k} q, iq \rangle - \langle j_{Z_k} q, iq \rangle \langle iq,iq \rangle = 0 
\end{equation*}
Thus $\eta$ canonically induces a $\mfz$-valued 1-form $\bar{\eta}$ on $\C P^n$. \\
Let $\Pi: S^{2n+1} \ra \C P^n$ denote the Hopf projection. Then we have: 
\begin{equation*}
\quad \eta = \Pi^* \bar{\eta}
\end{equation*}
In particular we have: \\
\begin{equation*}
\quad \bar{\eta}([X])_{[q,r,s]} = \eta(X)_{(q,r,s)}
\end{equation*}
where $[X]$ is any vector field on $\C P^n$ and $X$ its $S^1$-invariant Hopf-horizontal lift to $S^{2n+1}$. \\
By the above calculation, we have $\bar{\lambda}_h=|r|^2|s|^2\bar{\eta}$. Thus, up to the factor $|r|^2|s|^2$, the 1-form $\bar{\eta}$ on $\C P^n$ is equal to $\bar{\lambda}_h$. Since $\bar{\lambda}_h$ is admissible, so is $\bar{\eta}$.
\end{proof}

\begin{notation}
Let $(M,g_0)= (\C P^n, g_{FS})$. We recall that $g_{FS}$ denotes the Fubini-Study metric on $\C P^n$, obtained as the Riemannian submersion metric associated to the standard metric $g_{eucl}$ on $S^{2n+1}$ via the Hopf projection. \\
Given a linear map $j:\mfz \ra \mfs \mfu (n-1)$, let $\bar{\eta}$ be the corresponding admissible 1-form on $\C P^n$ as in proposition \ref{defeta} and define the associated Riemannian metric $g_{\bar{\eta}}$ on $\C P^n$ as in definition \ref{admissible}. \\
Explicitly we have, for vector fields $[X], [Y]$ on $\C P^n$ and their $S^1$-invariant Hopf-horizontal lifts $X,Y$ to the sphere:
\begin{eqnarray*}
& g_{\bar{\eta}}([X],[Y]) \\
=& g_{FS}([X]+\bar{\eta}[X]^*, [Y]+\bar{\eta}[Y]^*) \\
=& \langle X+\eta^1(X)Z_{h,1}^*+\eta^2(X)Z_{h,2}^*, Y+\eta^1(Y)Z_{h,1}^*+\eta^2(Y)Z_{h,2}^* \rangle \\
\end{eqnarray*}
\end{notation}

\begin{remark}
This new metric ${g}_{\bar\eta}$ is again T-invariant as $\bar\eta$ is T-invariant.  As $\bar{\eta}$ is an admissible 1-form on $\C P^n$ we can apply proposition \ref{riemsubmersion}, so we have that the torus acts isometrically on $(M, {g}_{\bar\eta})$ and that the induced Riemannian submersion metric $g_{\bar{\eta}}^T$ on $\hat{M} \slash T$ is equal to the metric $g_{0}^T$.
\end{remark}

\begin{theorem}
Let $j,j': \mfz \cong \R^2 \ra \mfs \mfu (n-1)$ be two linear maps,
and let ${g}_{\bar\eta}, {g}_{\bar\eta'}$ be the associated pair of
Riemannian metrics on $\C P^{n}$ as above. If $j \sim j'$, then
the Riemannian manifolds $(\C P^{n},{g}_{\bar\eta})$ and $(\C
P^{n},{g}_{\bar\eta'})$ are isospectral.
\end{theorem}
\begin{proof}
We will show that the condition:\\
\begin{itemize}
\item[$(I)$] For every $\mu \in \mcL^*$ there exists a T-equivariant $F_{\mu}
\in Isom(\C P^{n},g_{FS})$ which satisfies $\mu \circ \bar{\eta} = F_{\mu}^*(\mu \circ \bar{\eta}')$.
\end{itemize}
is fulfilled. We can then conclude by theorem \ref{isomlambda} that the two metrics on $\C P^{n}$ are isospectral. \\
We will at first give an explicit isometry fulfilling the analogon of $(I)$ on the sphere $(S^{2n+1}, g_{eucl})$ and then show it induces the desired isometry on $(\C P^{n}, g_{FS})$.\\
Fix an arbitrary $\mu \in \mcL^{*}$. As we have $\mcL^{*} \subset \mfz^{*}$, we can set $Z \in \mfz$ to be the vector corresponding to $\mu$ under the identification of $\mfz$ with $\mfz^*$ associated with the basis $\{Z_1,Z_2 \}$. Let $A_Z \in SU(n-1)$ be as in definition \ref{defj}. We set $G_{\mu}:=(A_Z, Id) \in SU(n-1) \times \{Id\} \subset SU(n+1)$. Thus $G_{\mu}$ defines an isometry of the sphere $(S^{2n+1}, g_{eucl})$. This isometry $G_{\mu}$ satisfies: \\
\begin{eqnarray*}
& G_{\mu}^{*}(\mu \circ \eta')_{p}(X)  \\
=& (\mu \circ \eta')_{(A_Z p)}(A_Z X) \\
%%%%%%%%%
=& |q|^2\langle j'_Z A_Z q, A_Z X_q \rangle 
+ \langle j'_Z A_Z q, iA_Z q \rangle \langle i A_Z q, A_Z X \rangle \\
%%%%%%%%%%
=& |q|^2\langle \bar A_Z^T j'_Z A_Z q, X_q \rangle 
+ \langle \bar A_Z^T j'_Z A_Z q, iq \rangle \langle iq, X \rangle \\
%%%%%%%%
=& |q|^2\langle j_Z q, X_q \rangle 
+ \langle j_Z q, iq \rangle \langle iq, X \rangle \\
%%%%%%%
=& (\mu \circ \eta)_{p}(X) 
\end{eqnarray*}
Since $G_{\mu} \in SU(n+1)$ it induces an isometry of $(\C P^n, g_{FS})$. Denote this isometry by $F_{\mu}$. Then the equation 
\begin{equation*}
\mu \circ \eta = G_{\mu}^*(\mu \circ \eta') \quad \textit{on } S^{2n+1}
\end{equation*}
implies 
\begin{equation*}
\mu \circ \bar{\eta} = F_{\mu}^*(\mu \circ \bar{\eta}') \quad \textit{on } \C P^n
\end{equation*}
because of $\eta^{(')} = \Pi^* \bar{\eta}^{(')}$ and  $\Pi \circ G_{\mu}= F_{\mu} \circ \Pi$ where $\Pi$ denotes the Hopf projection. \\
\end{proof}

\pagebreak
\section{Nonisometry}
%***************************************************************

\begin{notation}
\label{connectionform} 
\begin{tabular}{r p{12.5cm}}
\textup{ 1.} & Let $(M,g_0) = (\C P ^{n}, g_{FS})$, where all representatives in $\C^{n+1}$ are choosen with norm one. Let $[p]=[q,r,s]$ be a point in $\C P ^{n}$, where $q$ stands for the first $n-1$ components and $r$ and $s$ for the last two. \\
\end{tabular}
\begin{enumerate}
\setcounter{enumi}{1}
\item  Let $\hat M = \{ [ q,r,s ] \in \C P^{n} \mid q,r,s \neq 0 \}$, $\hat M$ is open and dense in $\C P^{n}$ and $T$ acts freely on $\hat M$, that is $\hat M$ can be seen as a principal $T$-bundle. \\ 
\item Recall that the torus action and the action induced by the Hopf fibration commute on $\C^{n+1}$ and a fortiori on the sphere. Setting $a=|r|$ and $b=|s|$ we see that $\hat M \slash T$ can be identified with 
\begin{equation*}
\big\{ ([q],a,b) \in  \C P^{n-2} \times \R \times \R \textit{ } \big| \textit{ } a,b,|q|>0, a^2+b^2 +|q|^2 = 1 \big\}
\end{equation*}
By $|q|$ we mean the norm of a representative of $[q]$. The metric induced by $g_0^{T}$ on the copy of $\C P^{n-2}$ with fixed values of $a$ and $b$ is $|q|^2g_{FS}$, that is a scalar multiple of the standard Fubini-Study metric. \\
\item As $\hat M$ is a principal T-bundle we have a $\mfz$-valued connection form $\omega_{\bar{\eta}}$ that is associated to the $\mfz$-valued 1-form $\bar{\eta}$. In any point $[p]$ in $\hat M$ we can decompose the tangent space into the flow of the torus action $\mfz^*$ as in notation \ref{VF*} and its orthogonal complement with respect to the metric $g_{\bar{\eta}}$. The connection form $\omega_{\bar{\eta}}$ assigns to each vector in the tangent space of $\hat M$ its component in $\mfz^*$. In other words we have $\omega_{\bar{\eta}}(Z^*)=Z$ for all $Z \in \mfz$ and $\omega_{\bar{\eta}}([X])=0$ for all $[X]$ that are orthogonal to $\mfz^*$ with respect to the metric $g_{\bar{\eta}}$ where $Z^*$ now denotes the vector field on $\hat{M}$ as in notation \ref{VF*}. Let $\omega_0$ denote the connection form associated with $g_{0}$. Then we have $\omega_{\bar{\eta}}=\omega_0 + \bar{\eta}$ by the defintion of $g_{\bar{\eta}}$.\\
\item Let $\Omega_{\bar{\eta}}$ denote the curvature form on $\hat M /T$ associated with the connection form $\omega_{\bar{\eta}}$. In the general case we have $\pi^* \Omega(X,Y)=d\omega (X,Y) + \frac{1}{2}[\omega(X), \omega(Y) ]$ where $\pi^*$ denotes the pullback of the projection $\pi:\hat{M} \ra \hat{M}\slash T$ and $[,]$ is the Lie bracket on $\mfz$. As $T$ is abelian this second term is zero and we get $\pi^* \Omega_{\bar{\eta}}=d\omega_{\bar{\eta}} $. \\
\end{enumerate}
\end{notation}

\begin{notation}
\label{Tpreserving}
We will say a diffeomorphism $F : M \ra M$ is T-preserving if conjugation by F preserves the torus seen as a subgroup of Diffeo(M). If $F$ is T-preserving then we denote by $\Psi_F$ the automorphism of $\mfz$, the Lie algebra of the torus, induced by conjugation by $F$. Note that $F_*(Z^*)=\Psi_F(Z)^* \circ F$ where $Z \in \mfz$ and $Z^*$ denotes the corresponding vector field on $\C P^n$ as in notation \ref{VF*}.
\end{notation}

\begin{definition}
Let $Aut^T_{g_0}(M)$ denote the group of all diffeomorphisms $F: M \ra M$ that
\begin{enumerate}
\item are T-preserving, 
\item preserve the $g_0$-norm of vectors tangent to the T-orbits and
\item induce an isometry of $(\hat{M} \slash T, g_0^T)$.
\end{enumerate}
Let $\overline{Aut}^T_{g_0}(M)$ denote the group of induced isometries $\bar{F}$ of $(\hat{M} \slash T, g_0^T)$. 
\end{definition}

\begin{notation}
For $|q|,a,b > 0$ let 
\begin{equation*}
M_{a,b}:=\big\{ [q,r,s] \in \hat{M} \textit{ }\big| \textit{ } |r|^2=a^2, |s|^2=b^2\big\}
\end{equation*}
Note that each $M_{a,b}$ is a submanifold, it is exactly the T-orbit of one copy of a $\C P^{n-2}$ on the q-coordinates and therefore T-invariant. $\hat{M}$ is the disjoint union of all $M_{a,b}$ with $0 < a^2+b^2 < 1$.
\end{notation}

\begin{lemma}
\label{Maberhalten}
Each $F \in Aut^T_{g_0}(M)$ preserves the set $M_{a,b} \cup M_{b,a}$ for any pair $a,b$.
\end{lemma}
\begin{proof}
As $F$ is T-preserving we get that $F$ preserves the set $\hat{M}$ as well as the set $M \setminus \hat{M}$. $F$ also fixes the set of all points where the torus orbit degenerates to a single point, this set consists of the two single points were either q and r or q and s are zero and of $M_0=\{ [q,r,s] \in M \mid r=s=0 \}$. As $M_0$ is a copy of $\C P^{n-2}$ and thus higher dimensional than the two single points it has to be preserved, too.\\
Next we will show that $F$ preserves the sets of all points with a fixed value of $a^2+b^2$ in $\hat{M}$. We will call this set $M_{a^2+b^2}$. It is a submanifold of $\hat{M} $ and $\hat{M}$ is the disjoint union of all $M_{a^2+b^2}$ with $0 < a^2+b^2 < 1$. \\
Note that $M_{a^2+b^2}$ is invariant under T, thus $M_{a^2+b^2} \slash T$ is a submanifold of $\hat{M} \slash T$. It consists of the set of all points in $\hat{M} \slash T$ with a fixed value of $a^2+b^2$. 
Let $M_{a^2+b^2} \slash T$ and $M_{a'^2+b'^2} \slash T$ denote two such sets. Then $M_{a^2+b^2} \slash T$ consists exactly of those points in $\hat{M}\slash T$ which have $g_0^T$-distance $d:=|\arccos \sqrt{1-a^2-b^2} - \arccos \sqrt{1-a'^2-b'^2}|$ to $M_{a'^2+b'^2} \slash T$. This can be seen by lifting to the sphere in two steps. On $\C P^n$ the assertion is equivalent to saying that the distance between $M_{a^2+b^2}$ and $M_{a'^2+b'^2}$ is equal to $d$ with respect to the Fubini-Study metric. Lifting to the sphere we have to prove that the distance between the two sets $\big\{ (q,r,s) \in S^{2n+1} \textit{ }\big|\textit{ } |r|^2+|s|^2=a^2+b^2 \big\}= \big\{ (q,r,s) \in S^{2n+1} \textit{ }\big|\textit{ } |q|^2=1-a^2-b^2 \big\}$ and $\big\{ (q,r,s) \in S^{2n+1} \textit{ }\big|\textit{ } |q|^2=1-a'^2-b'^2 \big\}$ is again equal to $d$ with respect to the standard metric $g_{eucl}$. An elementary calculation shows this to be true. \\
The set $M_0 \slash T$ does not lie in $\hat{M}\slash T$ but it lies in its completion as a metric space with respect to the distance $d_{g_0^T}$. Thus, passing to the limit for $a^2+b^2 \ra 0$, we conclude that $M_{a^2+b^2} \slash T$ consists of exactly those points in $\hat{M}\slash T$ which have distance $\arccos \sqrt{1-a^2-b^2}$ to $M_0 \slash T$ in the completion of the metric space $(\hat{M} \slash T, d_{g_0^T})$. Since $\bar{F}$ is an isometry of $(\hat{M} \slash T, d_{g_0^T})$, and since $F$ preserves $M_0$, it follows that $\bar{F}$ preserves $M_{a^2+b^2} \slash T$. This implies that $F$ preserves $M_{a^2+b^2}$.\\
That $F$ preserves the submanifolds $M_{a^2+b^2}$ of $\hat{M}$ means, in other words, that the function $\hat{M} \ni [q,r,s] \mapsto |r|^2+|s|^2 \in \R$ is invariant under $F$. Moreover, consider the $g_0$-area $A([q,r,s])$ of the T-orbit of a point $[q,r,s] \in \hat{M}$. We have:
\begin{eqnarray*}
A([q,r,s])^2 =& ||Z^*_{h,1} \wedge Z^*_{h,2}||^2_{(q,r,s)} \\
=& (||Z_{h,1}^*||^2||Z_{h,2}^*||^2 - \langle Z_{h,1}^*, Z_{h,2}^* \rangle ^2 )_{(q,r,s)} \\
=& |r|^2(1-|r|^2)|s|^2(1-|s|^2)-|r|^4|s|^4 \\
=& |r|^2|s|^2(1-|r|^2-|s|^2) \\
\end{eqnarray*}
Since F preserves the $g_0$-norms of vectors tangent to the T-orbits, it has to preserve $A([q,r,s])$. Thus not only the function $[q,r,s] \mapsto |r|^2+|s|^2$ but also the function $[q,r,s] \mapsto |r|^2|s|^2$ is invariant under $F$. This means that $F$ can map a point in $M_{a,b}$ either to a point in $M_{a,b}$ or to a point in $M_{b,a}$. Therefore the set $M_{a,b} \cup M_{b,a}$ is invariant under $F$, as claimed.
\end{proof}

\begin{proposition}
Let $D$ be defined as $D:=\{ \Psi_F \mid F \in Aut^T_{g_0}(M)\}$. Then \\
\begin{equation*}
D \subset \left\{ \left(\begin{array}{cc} \pm 1 & 0 \\ 0 & \pm 1 \end{array}\right),
        \left(\begin{array}{cc} 0 & \pm 1  \\ \pm 1 & 0 \end{array}\right) \right\} =\mcE
\end{equation*}
In other words the group of induced automorphisms of $\mfz$ is contained in the subgroup of the automorphism group of $\mfz$ of order eight used in the definition \ref{defj} to define non equivalent maps $j, j'$. 
\end{proposition}
\begin{proof}
$D$ is discrete as it preserves the lattice $\mcL$ associated to the torus. In fact it is a subgroup of the discrete group $\big\{ \Psi \in Aut(\mfz) \hspace{0.8mm} \big| \hspace{0.8mm} \Psi(\mcL)=\mcL \big\} \cong \big\{ A \in M(2,\Z) \hspace{0.8mm} \big| \hspace{0.8mm} det(A)=\pm 1 \big\}$. \\
Let $F \in Aut^T_{g_0}(M)$. We will restrict our attention to one of the submanifolds $M_{a,a}$ of $\hat{M}$. 
We have that $F$ preserves $M_{a,a}$ by lemma \ref{Maberhalten} in the special case of $a=b$.\\
Let $[p]=[q,r,s] \in M_{a,a}$ with $0 < a < \frac{1}{\sqrt{2}}$. 
Let $(Z_k^*)_{[p]}$, $k=1,2$, denote the vector fields induced by the torus action as in notation \ref{VF*} on $\C P^n$ in the point $[p]$.
Then we have 
\begin{eqnarray*}
\varangle ((Z_1^*)_{[p]}, (Z_2^*)_{[p]})
=& \arccos &\frac{\langle (Z_1^*)_{[p]}, (Z_2^*)_{[p]} \rangle}{||(Z_1^*)_{[p]}||\cdot ||(Z_2^*)_{[p]}||} \\
=& \arccos &\frac{\langle Z_{h,1}^*, Z_{h,2}^* \rangle_p}{||Z_{h,1}^*||_p \cdot ||Z_{h,2}^*||_p} \\
=& \arccos &\frac{-a^2 \cdot a^2}{a\sqrt{1-a^2} \cdot a\sqrt{1-a^2}} \\
=& \arccos &\frac{-a^2}{1-a^2} \\
\end{eqnarray*}
For sufficiently small $a>0$ this angle will be greater than $\frac{\pi}{3}$. Choose an $a \in (0, \frac{1}{\sqrt{2}})$ with this property. \\
Consider the T-orbit of $[p]$. The metric induced there by $g_{\bar{\eta}}$ is the same as the one induced by $g_0$ by proposition \ref{riemsubmersion}. The T-orbit endowed with this metric is a flat torus isometric to $span \{(Z_1^*)_{[p]}, (Z_1^*)_{[p]} \}$ divided by the lattice generated by $(Z_1^*)_{[p]}$ and  $(Z_2^*)_{[p]}$. \\
Since these two vectors are of equal length and the angle between them is in $(\frac{\pi}{3}, \frac{\pi}{2})$, they are, together with their negatives, exactly the shortest vectors in this lattice. In consequence the flow lines induced by $(Z_1^*)_{[p]}$ and  $(Z_2^*)_{[p]}$ are the shortest geodesic loops in the T-orbit of $[p]$.\\
As $F$ preserves $M_{a,a}$, it follows that $(F_*)_{[p]} Z_k^* \in \{ \pm (Z_1^*)_{F([p])}, \pm (Z_2^*)_{F([p])}\}$ for $k=1,2$. But we also have $(F_*)_{[p]} Z_k^*= (\Psi_F(Z_k))^*_{F([p])}$ by notation \ref{Tpreserving}. Thus we have $\Psi_F(Z_k) \in \{ \pm Z_1, \pm Z_2 \} $ for $k=1,2$ and the statement follows. \\
\end{proof}

We need the exterior derivative $d\bar\eta$ of $\bar\eta$ for the non isometry proof. 
\begin{lemma}
\label{deta}
We have 
\begin{eqnarray*}
d\bar\eta^k([X],[Y])_{[q,r,s]}  =& d\eta^k(X,Y)_{(q,r,s)} \\
=& 2\langle X_q,q \rangle \langle j_{Z_k} q, Y_q \rangle 
- 2 \langle Y_q,q \rangle \langle j_{Z_k} q, X_q \rangle \\
&+2|q|^2\langle j_{Z_k} X_q,Y_q \rangle \\
&-2\langle j_{Z_k} X_q,iq \rangle\langle iq, Y_q \rangle 
+ 2\langle j_{Z_k} Y_q,iq \rangle\langle iq, X_q \rangle \\
& -2\langle j_{Z_k} q, iq \rangle \langle iX_q,Y_q \rangle
\end{eqnarray*}
Here $X$ and $Y$ again denote the $S^1$-invariant Hopf-horizontal lifts of $[X]$ and $[Y]$. We denote the components of the vector field $X$ in the point $(q,r,s)$ by $X_q, X_r$ and $X_s$ and similarly for $Y$.
\end{lemma}

\begin{proof}
By the proof of proposition \ref{defeta} we have $\eta = \Pi^* \bar\eta $.
This implies $d\eta = \Pi^* d\bar\eta$; that is,
\begin{equation*}
d\bar\eta^k([X],[Y])_{[q,r,s]} = d\eta^k(X,Y)_{(q,r,s)} 
\end{equation*}
The remainder of the proof is a straightforward calculation. \\
Recall that 
\begin{equation*}
\eta^k(X)_p =|q|^2\langle j_{Z_k} q, X_q \rangle - \langle j_{Z_k} q, iq \rangle \langle iq,X_q \rangle 
\end{equation*}
for $k=1,2$; $p=(q,r,s) \in S^{2n+1}$ and $X \in T_pS^{2n+1}$. \\
Extend $\eta^k$ to a 1-form on $\R^{2n+2} \cong \C^{n+1}$ by the same formula. Let $X,Y \in T_pS^{2n+1}$, and extend them to constant vector fields on $\C^{n+1}$. Then $[X,Y]=0$, hence \\ 
\begin{eqnarray*}
d\eta^k (X,Y)_p=& X_p(\eta^k(Y))-Y_p(\eta^k(X))-\eta^k([X,Y])_p \\
=& \frac{d}{dt}\big|_{t=0}\big( \eta^k(Y)_{p+tX} - \eta^k(X)_{p+tY} \big)\\
%%%%%%%%%
=& 2\langle X_q,q \rangle \langle j_{Z_k} q, Y_q \rangle 
 + |q|^2\langle j_{Z_k} X_q,Y_q \rangle \\
& - \langle j_{Z_k} X_q,iq \rangle\langle iq, Y_q \rangle 
- \langle j_{Z_k} q,iX_q \rangle\langle iq, Y_q \rangle \\
& -\langle j_{Z_k} q, iq \rangle \langle iX_q,Y_q \rangle \\
& -2\langle Y_q,q \rangle \langle j_{Z_k} q, X_q \rangle 
 - |q|^2\langle j_{Z_k} Y_q,X_q \rangle \\
&+ \langle j_{Z_k} Y_q,iq \rangle\langle iq, X_q \rangle 
+ \langle j_{Z_k} q,iY_q \rangle\langle iq, X_q \rangle \\
& +\langle j_{Z_k} q, iq \rangle \langle iY_q,X_q \rangle \\
%%%%%%%%%
=& 2\langle X_q,q \rangle \langle j_{Z_k} q, Y_q \rangle 
- 2 \langle Y_q,q \rangle \langle j_{Z_k} q, X_q \rangle \\
&+2|q|^2\langle j_{Z_k} X_q,Y_q \rangle \\
&-2\langle j_{Z_k} X_q,iq \rangle\langle iq, Y_q \rangle 
+ 2\langle j_{Z_k} Y_q,iq \rangle\langle iq, X_q \rangle \\
& -2\langle j_{Z_k} q, iq \rangle \langle iX_q,Y_q \rangle \\
\end{eqnarray*}
\end{proof}

\begin{remark}
\label{restriction}
We will restrict our attention to one of the submanifolds $L:=M_{a,a}$ and observe what the connection form and the induced curvature form will look like. We will denote the restriction of the metric $g_{\bar{\eta}}$ to $L$ by $g_{\bar{\eta}}$ again. \\
$L$ is exactly the torus orbit of one copy of a $\C P^{n-2}$ on the q-coordinates. It is a submanifold and a principal T-bundle so we could define a connection form directly without considering the ambient manifold $\hat{M}$. We could also look at the restriction of the connection form of $\hat{M}$ to $L$. The definition of a connection form shows that we would get the same $\mfz$-valued 1-form on $L$ in both cases, we will denote it by $\omega_{\bar{\eta}}^L$. \\
We again have two possibilities for the induced curvature form on $L \slash T$ defining it as either
the curvature form associated with $\omega_{\bar{\eta}}^L$, or the restriction of the curvature form $\Omega_{\bar{\eta}}$ to $L \slash T$. But since restriction (i.e., pullback by inclusion) commutes with exterior derivation, we have $d(\omega_{\bar{\eta}}^L)=(d\omega_{\bar{\eta}})^L$; hence also the induced forms on $L \slash T$ are equal. \\
Thus the equation $\pi^* \Omega_{\bar{\eta}}=d\omega_{\bar{\eta}}$ from notation \ref{connectionform} carries over to $L$ and we have 
\begin{eqnarray*}
\pi^* \Omega^L_{\bar{\eta}}=d\omega^L_{\bar{\eta}} =& d\omega^L_0+d\bar{\eta}^L \\
\pi^* \Omega_{\bar{\eta}}^L =& \pi ^* \Omega_0^L + d\bar{\eta}^L 
\end{eqnarray*}
$L$ is a submanifold of codimension 2. For any point $[p]=[q,r,s]$ in $L$ its tangent space $T_{[p]}L$ consists of those $[X] \in T_{[p]}\hat{M}$ with
\begin{equation*}
\langle X_q, q \rangle=0 \hspace{1cm} \langle X_r, r \rangle=0 \hspace{1cm} \langle X_s, s \rangle=0
\end{equation*}
for the Hopf-horizontal lift $X$ of $[X]$ to $p=(q,r,s)$.
The sum of these three equations is $\langle X, p \rangle=0$, which is true for all tangent vectors $X \in T_pS^{2n+1}$, so we have only two independent equations. Inserting these into our formula for $d\bar\eta$ we get
\begin{eqnarray*}
d\bar\eta^k([X],[Y])_{[q,r,s]}=& 2|q|^2\langle j_{Z_k} X_q,Y_q \rangle \\
&-2\langle j_{Z_k} X_q,iq \rangle\langle iq, Y_q \rangle 
+ 2\langle j_{Z_k} Y_q,iq \rangle\langle iq, X_q \rangle \\
& -2\langle j_{Z_k} q, iq \rangle \langle iX_q,Y_q \rangle
\end{eqnarray*}
for $k=1,2$ and for all $[X], [Y] \in T_{[q,r,s]}L$, where $X$ and $Y$ again denote the Hopf-horizontal lifts of $[X]$ and $[Y]$ to $T_{(q,r,s)}S^{2n+1}$. This formula describes the $\mfz$-valued 1-form $d\bar{\eta}^L$ on $L$.
\end{remark}

\begin{remark}
\label{LKrmnull}
We will compute $\omega_0^L$ and $d\omega_0^L$. Let again $[p]=[q,r,s] \in L = M_{a,a}$ with $0 < a < \frac{1}{\sqrt{2}}$ and $[X]\in T_{[p]}L \subset T_{[p]}\hat{M}$. Let $X$ be the Hopf-horizontal lift of $[X]$ to $T_p S^{2n+1}$. By definition of the connection form $\omega_0$, we have 
\begin{eqnarray*}
\omega_0^L([X])_{[p]} =& \omega_0([X])_{[p]} \\
=& g_0([X], \frac{[Z_{h,1}^*]}{||[Z_{h,1}^*]||^2})Z_1 
+ g_0([X], \frac{[Z_{h,2}^*]}{||[Z_{h,2}^*]||^2})Z_2 \\
=& \big\langle X, \frac{Z_{h,1}^*}{||Z_{h,1}^*||^2} \big\rangle Z_1 
+ \big\langle X, \frac{Z_{h,2}^*}{||Z_{h,2}^*||^2} \big\rangle Z_2 \\
=& \big\langle X, \frac{(0,i r,0)- |r|^2 ip}{|r|^2 (1-|r|^2)} \big\rangle Z_1 
+ \big\langle X, \frac{(0,0,i s)- |s|^2 ip}{|s|^2 (1-|s|^2)} \big\rangle Z_2 \\
=& \big\langle X_r, \frac{i r}{|r|^2 (1-|r|^2)} \big\rangle Z_1 
+ \big\langle X_s, \frac{i s}{|s|^2 (1-|s|^2)} \big\rangle Z_2 \\
\end{eqnarray*}
In the last equation we used $\langle X, ip \rangle =0 $. \\
Thus one could see $\omega_0^L$ as the sum of two 1-forms. The first one depends only on the $r$ component of the point $[p]$ and applies to vector fields that have only the component $X_r$ in $T_{[p]}L$, the second one depends only on $s$.
As we are looking at $L$, we have $\langle X_r,r \rangle=0$ for all $[X] \in T_{[p]}L$, that means this component of $T_{[p]}L$ has only one real dimension. In consequence the r-part of $d\omega_0^L$, a 2-form on this space, has to be zero and similarly for the s-part. Hence we have
\begin{eqnarray*}
d\omega_0^L=0
\end{eqnarray*}
and consequently
\begin{eqnarray*}
\Omega_0^L=0
\end{eqnarray*}
This will be useful in the proof of theorem \ref{nonisothm}.
\end{remark}

We now have the means to formulate the two theorems that will grant us non isometry of the constructed metrics. The first is a 
general result working for all metrics constructed using the admissible 1-forms introduced in section 2. The second one is the application of this theorem to our setting of complex projective space. \\

\begin{theorem}
\cite{Schueth01b}
Let $\eta$, $\eta'$ be admissible 1-forms on M such that
$\Omega_{\eta}$ and $\Omega_{\eta'}$ satisfy the condition: \\
(G) No nontrivial 1-parameter group in $\overline{Aut}_{g_0}^T(M)$ preserves $\Omega_{\eta}$. \\
Furthermore, assume that \\
(N) $\Omega_{\eta} \notin D \circ \overline{Aut}_{g_0}^T(M)^*
\Omega_{\eta'}$. \\
Then the manifolds $(M,g_{\eta})$ and $(M,g_{\eta'})$ are
not isometric.
\end{theorem}

We shortly sketch the structure of the proof: \\
Condition (G) implies that $T$ is a maximal torus in $Isom(M, g_{\eta})$. If $T$ were not maximal then the additional dimension could be used to build a 1-parameter group as in condition (G). \\
Now suppose $F$ were an isometry from $(M,g_{\eta})$ to $(M,g_{\eta'})$. We denote the induced isometry of $(\hat{M} \slash T, g_0^T)$ by $\bar{F}$. All maximal tori in $Isom(M, g_{\eta})$ are conjugate and as $T$ is one of them we can assume $F$ to be T-preserving. This implies $F^*\omega_{\eta'}=\Psi_F \circ \omega_{\eta}$, which in turn implies $\bar{F}^*\Omega_{\eta'}=\Psi_F \circ \Omega_{\eta}$. This contradicts (N). \\

\begin{theorem}
\label{nonisothm}
Let  $j$, $j' : \mfz \ra \mfs \mfu(n-1)$ be two linear maps and let $\bar\eta$, $\bar\eta'$ be the induced 1-forms on $M=\C P^n$ as in proposition \ref{defeta}, then
\begin{enumerate}
\item If $j$, $j'$ are not equivalent, then $\Omega_{\bar\eta}$ and
$\Omega_{\bar\eta'}$ satisfy condition \\
(N) $\Omega_{\bar\eta} \notin D \circ \overline{Aut}_{g_0}^T(M)^*
\Omega_{\bar\eta'}$

\item If $j$ is generic, then $\Omega_{\bar\eta}$ satisfies
condition \\
(G) No nontrivial 1-parameter group in $\overline{Aut}_{g_0}^T(M)$ preserves $\Omega_{\bar\eta}$ \\
\end{enumerate}
In particular, if $j$ and $j'$ are not equivalent and $j$ is generic, then the isospectral manifolds $(\C P^n, g_{\bar{\eta}})$ and $(\C P^n, g_{\bar{\eta}'})$ are not isometric.
\end{theorem}

\begin{proof}
\item 1. We will use an indirect proof. Suppose (N) were
not satisfied. Let $\Psi \in D$ and $\bar F \in \overline{Aut}_{g_0}^T(M)$ such
that \\
$(*) \qquad \Omega_{\bar\eta} = \Psi \circ \bar F ^*\Omega_{\bar\eta'} $\\
Let $\omega_{\bar\eta}^L$ and $\Omega_{\bar\eta}^L$ denote the induced connection and curvature forms on $L$, and similarly for $\bar\eta'$, where $L$ is again one of the submanifolds $M_{a,a}$ with $0 < a < \frac{1}{\sqrt{2}}$.
By remark \ref{restriction} and remark \ref{LKrmnull} we have 
\begin{equation*}
\pi^*\Omega^L_{\bar{\eta}}=d\bar{\eta}^L \quad \textit{ and } \quad \pi^*\Omega^L_{\bar{\eta}'}=d\bar{\eta}'^L
\end{equation*}
on $L \subset \hat{M} \subset \C P^n$.
Let $F$ be a map in $Aut^T_{g_0}(M)$ inducing $\bar{F}$. \\
Recall that $F$ preserves $L$ by lemma \ref{Maberhalten} applied to $a=b$. Thus the isometry $\bar{F}$ of $(\hat{M}\slash T, g_0^T)$ restricts to an isometry $\bar{F}^L$ of $(L \slash T, g_0^T)$. Equation $(*)$ now implies 
\begin{equation*}
d\bar{\eta}^L=\Psi \circ (\bar{F}^L \circ \pi)^* d\bar{\eta}'^L
\end{equation*}
Recall that $L$ is the torus orbit of one copy of a $\C P^{n-2}$ on the $q$-coordinates, thus $(L \slash T, g_0^T)$ is isometric to $(\C P^{n-2}, (1-2a^2)g_{FS})$.  Hence, $\bar{F}^L$ corresponds to a map $A \in Isom(\C P^{n-2}, g_{FS})=SU(n-1) \cup SU(n-1) \circ Q$, and $(\bar{F}^L \circ \pi )[q,r,s]=\pi[Aq, r,s]$ for all $[q,r,s] \in L$. Recall that $Q$ denotes complex conjugation. We write \\
\begin{tabular}{ccccl}
$ B:$ & $L$ & $\ra$ & L &\\
& $[q,r,s]$ & $ \mapsto$ & $[ Aq,r,s ]  $& if  $A \in SU(n-1)$ \\
& $[q,r,s]$ & $ \mapsto$ & $[ A\bar{q},\bar{r},\bar{s} ]$ & if $A \in SU(n-1)\circ Q $\\
\end{tabular} \\
It is easy to check that $B$ is well-defined. 
Then, by $\bar{F}^L \circ \pi = \pi \circ B$, we get \\
$(H) \qquad d\bar{\eta}^L=\Psi \circ B^*(d\bar{\eta}'^L)$ \\
The next step will be to use the formula for $d\bar\eta^L$ calculated in remark \ref{restriction} and show that this equation can only be fulfilled if $j$ and $j'$ are equivalent. \\
Before doing so, we will introduce a simplification. We will restrict our attention to tangent vectors $[X], [Y] \in T_{[q,r,s]}L$ with $\langle X_q, iq \rangle = 0 = \langle Y_q, iq \rangle$ for the Hopf-horizontal lifts $X,Y \in T_{(q,r,s)}S^{2n+1}$. For such $[X],[Y]$ the formula for $d\bar{\eta}$ from remark \ref{restriction} simplifies to 
\begin{equation*}
d\bar{\eta}^k([X],[Y])_{[q,r,s]}=2|q|^2 \langle j_{Z_k}X_q, Y_q \rangle 
- 2 \langle j_{Z_k}q,iq \rangle \langle iX_q, Y_q \rangle
\end{equation*}
Note that, no matter whether $A\in SU(n-1)$ or $A\in SU(n-1) \circ Q$, the condition $\langle X_q, iq \rangle = 0$ implies $\langle AX_q, iAq \rangle = \pm \langle AX_q, Aiq \rangle = \pm \langle X_q, iq \rangle = 0$, and similarly for $Y$. \\
Thus, the above simplified formula for $d\bar{\eta}^k([X],[Y])$ will also apply to $B_*[X]=(AX_q, \hdots )$ and $B_*[Y]=(AY_q, \hdots )$ in $T_{B[q,r,s]}L$.\\
As a last preparation, note moreover that for any $\mfz$-valued 1-form $\alpha$ on $L$ we have
\begin{equation*}
(\Psi \circ \alpha)^k=\langle \Psi(\alpha(.)), Z_k \rangle = \langle \alpha (.), ^t\Psi(Z_k) \rangle
\end{equation*}
for $k=1,2$, where $\langle , \rangle$ denotes the inner product with orthogonal basis $\{Z_1, Z_2 \}$ on $\mfz$, and $^t \Psi$ denotes the adjoint of $\Psi$ with respect to this inner product. \\
Thus, equation $(H)$ implies for all $[X], [Y] \in T_{[q,r,s]}L$ whose Hopf horizontal lifts to $T_{(q,r,s)}S^{2n+1}$ satisfy $\langle X_q, iq \rangle = 0= \langle Y_q, iq \rangle$ that we have
\begin{eqnarray*}
0= & 2|q|^2 \langle j_{Z_k}X_q, Y_q \rangle - 2 \langle j_{Z_k}q,iq \rangle \langle iX_q, Y_q \rangle \\
& - 2|Aq|^2 \langle j'_{^t\Psi(Z_k)}AX_q, AY_q \rangle + 2 \langle j'_{^t\Psi(Z_k)}Aq,iAq \rangle \langle iAX_q, AY_q \rangle \\
0= &|q|^2 \langle j_{Z_k}X_q, Y_q \rangle - |q|^2 \langle A^{-1}j'_{^t\Psi(Z_k)}AX_q, Y_q \rangle \\
& -  \langle j_{Z_k}q,iq \rangle \langle iX_q, Y_q \rangle 
+ \langle A^{-1}j'_{^t\Psi(Z_k)}Aq,iq \rangle \langle iX_q, Y_q \rangle\\
\end{eqnarray*}
because $A$ either commutes or anticommutes with $i$. \\
Letting $\nu^k := j_{Z_k}-A^{-1}j'_{^t \Psi(Z_k)} A$, we get
\begin{equation*}
0=|q|^2\langle \nu^k X_q, Y_q \rangle - \langle \nu^k q, iq \rangle \langle iX_q, Y_q \rangle
\end{equation*}
This equation holds for all $q \in S^{2n-3}_{\sqrt{1-2a^2}}$ and all $X_q, Y_q \in \C^{n-2}$ with 
$X_q, Y_q \perp span \{q, iq\}$, because all such $X_q, Y_q$ occur in Hopf-horizontal lifts of the form $(X_q, 0,0), (Y_q, 0,0) \in T_{(q,r,s)}S^{2n-1}$ of tangent vectors in $T_{[q,r,s]}L$. Note that our $n$ is at least 4 by section 2, so $n-2 \ge 2$, thus nonzero $X_q \perp span \{q, iq\}$ do exist. \\
In the particular case $Y_q = iX_q$, we get \\
$(\triangle) \qquad \frac{\langle i\nu^k X_q, X_q \rangle}{|X_q|^2} = \frac{\langle i\nu^k q, q \rangle}{|q|^2} $\\
for all $q \in S^{2n-3}_{\sqrt{1-2a^2}}$ and all $X_q \perp span\{q, iq\}$ in $\C^{n-2}$. Applying $(\triangle)$ to elements of an orthonormal basis of eigenvectors for the hermitian map $i\nu^k$, we see that all eigenvalues of $i\nu^k$ have to be equal.
Thus $i\nu^k$ and $\nu^k$ have to be scalar multiples of the identity. By definition $\nu^k \in \mfs \mfu (n-1)$, so $\nu^k$ has trace zero. This implies $\nu^k =0$ and we have shown that $j$ and $j'$ are equivalent. \\
Thus we have shown that if $j$ and $j'$ are not equivalent, then condition (N) is fulfilled. \\

\item 2. We will again use an indirect approach. Suppose there were a non trivial 1-parameter group 
$\bar{F}_t \subset \overline{Aut}_{g_0}^T(M)$ with $\bar F_t^* \Omega_{\bar\eta} = \Omega_{\bar\eta}$. We will again restrict our attention to one of the submanifolds $L$. There we get $\bar{F}_t^* \Omega^L_{\bar\eta} = \Omega^L_{\bar\eta}$. Using the same arguments as in the first part with $\Psi=id$ and $\bar{\eta}=\bar{\eta}'$, we get $j_Z=A_t^{-1}j_ZA_t$ for some nontrivial 1-parameter group $A_t \subset SU(n-1)$. This contradicts the genericity assumption
made on $j_Z$. \\
\end{proof}

\pagebreak
\section{Real projective space}
%********************************************************************************

We will present here isospectral metrics on real projective space. One can see a real projective space as a sphere with opposite points identified, using this construction the metrics on the sphere carry over directly to real projective space. We will at first present the key data of the isospectral metrics on the sphere constructed by Schueth in \cite{Schueth01b} and then show that everything is compatible with factoring by the antipodal map. \\

\begin{notation}
Let the sphere $S^{2m+1}$ be embedded in $\C^{m+1}=\C^m \oplus \C$. Denote a point on the sphere by $(p,q)$. Let the torus $T$ act on the sphere by
\begin{equation*}
exp(aZ_1+bZ_2) : (p,q) \mapsto (e^{ia}p,e^{ib}q)
\end{equation*}
for all $a,b \in \R$. Here $Z_1, Z_2$ denotes again the basis of $\mfz$, the Lie algebra of the torus $T$.
\end{notation}

\begin{notation}
\label{rpn}
Let $j(t) : \mfz \ra \mfs \mfu (m)$, again $m \ge 3$, be a family of linear maps that are isospectral, non equivalent and generic as in definition \ref{defj}. We then define the associated $\mfz$-valued 1-forms $\lambda(t)$ on $\C^m \oplus \C$ by 
\begin{equation*}
\lambda^k_{(p,q)}(X,U) := |p|^2 \langle j_{Z_k}p,X \rangle - \langle X,ip \rangle \langle j_{Z_k}p, ip \rangle \quad \textit{for } k=1,2
\end{equation*}
where $(p,q)$ is a point in $\C^m \oplus \C$ and $(X,U)$ a vector in the tangent space $T_p\C^m \oplus T_q\C$. Note the similarity to the 1-form $\eta$ defined in proposition \ref{defeta}. This stems from the fact that the 1-forms were orthogonalized to $ip$ in both cases. 
These 1-forms are restricted to the sphere, they are admissible as in definition \ref{admissible} and the associated metrics 
\begin{equation*}
g_{\lambda}(X,Y) := g_0(X + \lambda(X)^*, Y + \lambda(Y)^*)
\end{equation*}
are isospectral and non isometric (see proposition 3.2.5 and 4.3 in \cite{Schueth01b} for the proof).
\end{notation}

\begin{proposition}
The isospectral metrics on the sphere induce isospectral metrics on real projective space in a canonical way. In fact we have
\begin{equation*}
\lambda^k_{(p,q)}(X,U) = \lambda^k_{(-p,-q)}(-X,-U)
\end{equation*}
thus $\lambda$ is invariant under the antipodal map, hence induces a $\mfz$-valued 1-form $\bar{\lambda}$ on $\R P^{2m+1}$. With respect to the accordingly defined metric $g_{\bar{\lambda}}$ on $\R P^{2m+1}$, the projection $(S^{2m+1}, g_{\lambda}) \ra (\R P^{2m+1}, g_{\bar{\lambda}})$ is a Riemannian covering.
\end{proposition}
\begin{proof}
This can be seen by applying the definition given in notation \ref{rpn}, all the minus signs cancel out.
\end{proof}

It is possible to imitate the nonisometry proof for metrics on spheres in \cite{Schueth01b} to obtain a nonisometry proof for the isospectral metrics on $\R P^{2m+1}$.

\begin{remark}
The same observation applies to the pair of isospectral metrics on $S^5$ constructed in \cite{Schueth01b} as well as to the isospectral metrics on spheres constructed by Carolyn Gordon in \cite{Gordon01}, all induce isospectral metrics on real projective space.
\end{remark}

\pagebreak
%\printindex
\bibliographystyle{amsalpha}
% \appendix
\addcontentsline{toc}{section}{Literature}
\bibliography{literatur}
% \begin{thebibliography}{A}
% \end{thebibliography}

\end{document}